\documentclass[12pt,a4paper]{article}

\usepackage[OT1]{fontenc}
\usepackage[utf8]{inputenc}

\usepackage[english]{babel}
\usepackage{amsfonts}
\usepackage{amsthm}
\usepackage{graphicx}

\newtheorem{teo}{Theorem}[section]
\newtheorem{prop}[teo]{Proposition}
\newtheorem{lem}[teo]{Lemma}
\newtheorem{coro}[teo]{Corollary}
\newtheorem{defi}[teo]{Definitions}
\newtheorem{rem}[teo]{Remark}
\newtheorem{ejem}[teo]{Example}

\begin{document}

\title{\vspace*{0cm}Inequalities related to Bourin and Heinz means with a complex parameter}





\date{}
\author{T. Bottazzi, R. Elencwajg, G. Larotonda and A. Varela
\footnote{All authors supported by Instituto Argentino de Matem\'atica, CONICET and Universidad Nacional de General Sarmiento.}}

\maketitle

\setlength{\parindent}{0cm} 

\begin{abstract}
A conjecture posed by S. Hayajneh and F. Kittaneh claims that given $A,B$ positive matrices, $0\le t\le 1$, and any unitarily invariant norm it holds
$$|||A^tB^{1-t}+B^tA^{1-t}|||\leq|||A^tB^{1-t}+A^{1-t}B^t|||.
$$
Recently, R. Bhatia proved the inequality for the case of the Frobenius norm and for $t \in[\frac14;\frac34]$. In this paper, using complex methods we extend this result to complex values of the parameter $t=z$ in the strip $\{z \in\mathbb C: {\rm Re}(z) \in [\frac14; \frac34]\}$. We give an elementary proof of the fact that equality holds for some $z$ in the strip if and only if $A$ and $B$ commute. We also show a counterexample to the general conjecture by exhibiting a pair of positive matrices such that the claim does not hold for the uniform norm. Finally, we give a counterexample for a related singular value inequality given by $s_j(A^tB^{1-t}+B^tA^{1-t})\leq s_j(A+B)$, answering in the negative a question made by K. Audenaert and F. Kittaneh.
\end{abstract}

\section{Introduction}\label{intro1}

We begin this paper with some notations and definitions. The context here is the algebra of $n\times n$ complex entries matrices, but the proofs adapt well to other (infinite dimensional) settings in operator theory, so let us assume that $\mathcal A$ stands for an operator algebra with trace, for instance ${\mathcal A}=M_n(\mathbb C)$ with its usual trace, or $\mathcal A=B_2(H)$, the Hilbert-Schmidt operators acting on a separable complex Hilbert space with the infinite trace, or ${\mathcal A}=({\mathcal A},Tr)$ a $C^*$-algebra with a finite faithful trace.

\begin{defi} Let  $|||\cdot|||$ denote an unitarily invariant norm on $\mathcal A$, which we assume is equivalent to a symmetric norm, that is
$$
|||XYZ||||\le \|X\|_{\infty}|||Y|||\|Z\|_{\infty}
$$
whenever $Y\in \mathcal A$ (from now on $\|\cdot\|_{\infty}$ will denote the norm of the operator algebra). 

\medskip

For convenience we will use the notation $\tau(X)=Re\; Tr(X)$. Let $|X|=\sqrt{X^*X}$ stand for the modulus of the matrix or operator $X$, then the (right) polar decomposition of $X$ is given by $X=U|X|$ where $U$ is a unitary such that $U$ maps $Ran|X|$ into $Ran(X)$ and is the identity on $Ran|X|^{\perp}=Ker(X)$. Note that $\|X\|_2^2=Tr(X^*X)=Tr[|X|^2]$.
\end{defi}

\bigskip

Consider the inequality
\begin{equation}\label{ftrazaC}
\tau(A^zB^zA^{1-z}B^{1-z}) \le \tau(AB),
\end{equation}
for positive invertible operators $A,B>0$ in $\mathcal A$, and $z\in \mathbb C$. We introduce some notation regarding vertical strips in the complex plane: let
$$
\mathcal S_0=\{z\in\mathbb C:0\le {\rm Re}(z)\le 1\},\qquad \mathcal S_{1/4}=\{z\in\mathbb C:1/4\le {\rm Re}(z)\le 3/4\};
$$ 
we will study the validity of (\ref{ftrazaC}) in both $S_0$ and $S_{1/4}$.

\medskip

Intimately related to the expression above are the inequalities
\begin{eqnarray}\label{bvsh}
|||b_t(A,B)||||\le |||h_t(A,B)|||
\end{eqnarray}
and
\begin{eqnarray}\label{b}
|||b_t(A,B)|||\le |||A+B|||,
\end{eqnarray}
for positive matrices $A,B\ge 0$ in $\mathcal A$, where 
$$
b_t(A,B)=A^tB^{1-t}+B^tA^{1-t}\quad t\in [0,1];
$$
the name $b_t$ is due to Bourin, who conjectured inequality (\ref{b}) for $n\times n$ matrices in \cite{bourin}, and 
$$
h_t(A,B)=A^tB^{1-t}+A^{1-t}B^t\quad t\in [0,1]
$$
is named after Heinz, and the well-known \cite{heinz} inequality
$$
|||h_t(A,B)|||\le |||A+B|||
$$
carrying his name. 

\bigskip

Recently, S. Hayajneh and F. Kittanneh proposed in \cite{kitta} that the stronger (\ref{bvsh}) should also be valid in $M_n(\mathbb C)$; however, numerical computations (see Section \ref{contrajemplos}) show that, at least for \textit{the uniform norm}, this is false.

\bigskip

If we focus on the case $|||X|||=\|X\|_2=Tr(X^*X)^{1/2}$ (the Frobenius norm in the case of $n\times n$ matrices) and we write $h_t=h_t(A,B)$, $b_t=b_t(A,B)$, then 
\begin{eqnarray}
Tr|b_t|^2 &= &\tau(b_t^*b_t)=\tau(B^{1-t}A^t+A^{1-t}B^t)(A^tB^{1-t}+B^tA^{1-t})\nonumber\\
& = &\tau(B^{2(1-t)}A^{2t})+ \tau(A^{2(1-t)}B^{2t})+2\tau(A^tB^tA^{1-t}B^{1-t})\nonumber
\end{eqnarray}
where we have repeatedly used the ciclicity of $\tau$ (i.e. $\tau(XY)=\tau(YX)$) and the fact that $\tau(Z^*)=\tau(Z)$. Likewise
$$
Tr|h_t|^2 =\tau(B^{2(1-t)}A^{2t})+ \tau(A^{2(1-t)}B^{2t})+2\tau(AB).
$$

Thus, proving that $\|b_t\|_2\le \|h_t\|_2$ amounts to prove that
\begin{equation}\label{ftraza}
\tau(A^tB^tA^{1-t}B^{1-t}) \le \tau(AB),
\end{equation}
and in fact, it is clear that both inequalities are equivalent -as remarked in \cite{kitta}-. 

\section{Main results}

We will divide the problem in regions of the plane (or the line), and then we will also consider the possiblity of attaining the equality; we will see that this is only possible in the trivial case, i.e. when $A,B$ commute. We recall the generalized H\"older inequality, that we will use frequently: let $\frac{1}{p}+\frac{1}{q}+\frac{1}{r}=1$ for $p,q,r\ge 1$ and $X,Y,Z$ in ${\mathcal A}$, then
$$
\tau(XYZ)\le\|XYZ\|_1\le \|X\|_p\|Y\|_q\|Z\|_r.
$$
This is just a combination of the usual H\"older inequality together with
$$
\|XY\|_s\le \|X\|_p\|Y\|_q
$$
provided $s\ge 1$ and $\frac{1}{p}+\frac{1}{q}=\frac{1}{s}$ (see \cite{simon}, Theorem 2.8, for more details).

\subsection{The inequality in the strip ${\mathcal S}_{1/4}$}

We begin with an easy consequence of an inequality due to Araki-Lieb and Thirring.
\begin{lem}\label{arak}
If $A,B\ge 0$ and $r\ge 2$, then
$$
\|A^{1/r}B^{1/r}\|_r\le \tau(AB)^{1/r}.
$$
\end{lem}
\begin{proof}
Note that 
$$
\|A^{1/r}B^{1/r}\|_r^r=\tau([A^{1/r}B^{1/r}B^{1/r}A^{1/r}]^{r/2})=\tau([A^{1/r}B^{2/r}A^{1/r}]^{r/2})
$$
which, by the inequality of Araki-Lieb and Thierring (see \cite{araki}, and note that $r/2\ge 1$) is less or equal than
$$
\tau(A^{r/2r}B^{r2/2r}A^{r/2r})=\tau(A^{1/2}BA^{1/2}),
$$
which in turn equals $\tau(AB)$.
\end{proof}

Note that if we exchange the variables $z\mapsto 1-z$ and exchange the role of $A,B$, it suffices to consider half-strips or half-intervals around ${\rm Re}(z)=1/2$.

\bigskip

\begin{prop}\label{desistrip}
If $0 < A,B$ and $z\in \mathcal S_{1/4}$, then
$$
\tau(A^zB^zA^{1-z}B^{1-z}) \le \tau(AB).
$$
\end{prop}
\begin{proof}
Let $z=1/2+iy$, $y\in\mathbb R$ denote any point in vertical line of the complex plane passing through $x=1/2$. Then
\begin{eqnarray}
\tau(A^zB^zA^{1-z}B^{1-z})&=  & \tau(A^{iy}A^{1/2}B^{1/2}B^{iy}A^{-iy}A^{1/2}B^{1/2}B^{-iy})  \nonumber\\
&\le & \tau |A^{iy}A^{1/2}B^{1/2}B^{iy}A^{-iy}A^{1/2}B^{1/2}B^{-iy}|\nonumber\\
&\le & \|A^{iy}A^{1/2}B^{1/2}B^{iy}A^{-iy}\|_2\|A^{1/2}B^{1/2}B^{-iy}\|_2=\|A^{1/2}B^{1/2}\|_2^2\nonumber
\end{eqnarray}
by the Cauchy-Schwarz inequality and the fact that $A^{iy},B^{iy}$ are unitary operators. Then by the previous lemma,
$$
\tau(A^zB^zA^{1-z}B^{1-z})\le \tau(AB)^{2/2}=\tau(AB).
$$
Now consider $z=1/4+iy$, $y\in\mathbb R$, a generic point in the vertical line over $x=1/4$, then noting that $\frac{1}{4}+\frac{1}{4}+\frac{1}{2}=1$,
\begin{eqnarray}
\tau(A^zB^zA^{1-z}B^{1-z})&=&\tau( B^{1/4}A^{1/4}A^{iy}B^{iy}B^{1/4}A^{1/4}A^{-iy}A^{1/2}B^{1/2}B^{-iy})\nonumber\\
&\le & \|B^{1/4}A^{1/4}\|_4^2\|B^{1/2}A^{1/2}\|_2\le \tau(AB)^{2/4+1/2}=\tau(AB)\nonumber,
\end{eqnarray}
where we used again the previous Lemma and the generalized H\"older's inequality,
$$
\tau(XYZ)\le \|X\|_{p}\|Y\|_q\|Z\|_r
$$
whenever $p,q,r\ge 1$ and $\frac{1}{p}+\frac{1}{q}+\frac{1}{r}=1$.

By Hadamard's three-lines theorem, the bound $\tau(AB)$ is valid in the vertical strip $1/4\le {\rm Re}(z)\le 1/2$, since it holds in the frontier of the strip. Invoking the symmetry $z\mapsto 1-z$ and exchanging the roles of $A,B$ gives the desired bound on the full strip ${\mathcal S}_{1/4}=\{1/4\le {\rm Re}(z)\le 3/4\}$.
\end{proof}

Regarding the inequalities conjectured by Bourin et al., note that we can assume $A,B>0$: replacing $A$ with $A_{\varepsilon}=A+\varepsilon$ (and likewise with $B$), if the inequality (\ref{ftrazaC}) is valid for $A_{\varepsilon},B_{\varepsilon}$ then making $\varepsilon\to 0^+$ gives the general result: the following result that we state as corollary was recently obtained by R. Bhatia in \cite{bhatia} and we should also point the reader to the paper by T. Ando, F. Hiai, K. Okubo \cite{ando}.

\medskip

\begin{coro}\label{tbvsh}
For any $A,B\ge 0$ and any $t\in [1/4,3/4]$, 
$$
\|A^tB^{1-t}+B^tA^{1-t}\|_2\le \|A^tB^{1-t}+A^{1-t}B^t\|_2\le\|A+B\|_2.
$$
\end{coro}

\subsection{Inequality becomes equality}

Let us consider the special case when the inequality above becomes an equality. We begin with a lemma that we will use in several ocasions, and will be useful when we drop the assumption on nonsingularity of $A,B$.

\begin{lem}\label{conmu}
Let $A,B\ge 0$, and assume 
$$
\tau(A^{1/2}B^{1/2}A^{1/2}B^{1/2})=\tau(AB),
$$
or
$$
\|A^{1/4}B^{1/4}\|_4=\tau(AB)^{1/4}.
$$
In either case, $A$ commutes with $B$. 
\end{lem}
\begin{proof}
Name $X=A^{1/2}B^{1/2}$, and considering the inner product induced by $\tau$, $\langle X,Y\rangle =\tau(XY^*)$,
$$
\langle X,X^*\rangle=\tau(X^2)=\tau(A^{1/2}B^{1/2}A^{1/2}B^{1/2})=\tau(AB)=\tau(X^*X)=\|X\|_2^2=\|X\|_2\|X^*\|_2.
$$
But Cauchy-Schwarz inequality becomes an equality if and only if $X=\lambda X^*$ for some $\lambda >0$, and since both operators have equal norm ($=\|A^{1/2}B^{1/2}\|_2$), then $X=X^*$. This means 
$$
A^{1/2}B^{1/2}=B^{1/2}A^{1/2},
$$
and this implies that $A$ commutes with $B$. On the other hand, 
$$
\|A^{1/4}B^{1/4}\|_4^4=\tau((B^{1/4}A^{1/2}B^{1/4})^2)=\tau(A^{1/2}B^{1/2}A^{1/2}B^{1/2}),
$$
so what we have is just another way of writing the first equality condition.
\end{proof}

\begin{prop}\label{igualdad}
Let $A,B>0$ and assume that there is $z_0\in {\mathcal S}_{1/4}$ such that 
$$
\tau(A^{z_0}B^{z_0}A^{1-z_0}B^{1-z_0})=\tau(AB).
$$
Then $A$ commutes with $B$ and $\tau(A^{z}B^{z}A^{1-z}B^{1-z})=\tau(AB)$ for any $z\in \mathbb C$.
\end{prop}
\begin{proof}
First consider the case when equality is reached in an interior point of the strip ${\mathcal S}_{1/4}$. Note that by the maximum modulus principle, this would mean that the function 
$$f(z)=\tau(A^zB^zA^{1-z}B^{1-z})$$ is constant in the strip ${\mathcal S}_{1/4}$, in particular equality holds at $z_0=1/2$, and by the previous Lemma, $A$ commutes with $B$.

Now suppose equality is attained in the frontier, for instance at $z_0=1/4+iy$ for some $y\in\mathbb R$. Let $X=B^{1/4}A^{1/4}A^{iy}B^{iy}B^{1/4}A^{1/4}$, $Y=B^{1/2}B^{iy}A^{iy}A^{1/2}$. Then, if we go through the proof of Proposition \ref{desistrip} again, assuming equality
\begin{eqnarray}
\tau(AB)&=  & \tau(XY^*)=\langle X,Y\rangle \le \|X\|_2\|Y\|_2\nonumber\\
& \le& \|B^{1/4}A^{1/4}\|_4^2\|A^{1/2}B^{1/2}\|_2  \le \tau(AB).
\end{eqnarray}
Arguing as in the previous Lemma, there exists $\lambda>0$ such that $X=\lambda Y$,
$$
B^{1/4}A^{1/4}A^{iy}B^{iy}B^{1/4}A^{1/4}=\lambda B^{1/2}B^{iy}A^{iy}A^{1/2}.
$$
Cancelling $B^{1/4}$ on the left and $A^{1/4}$ on the right we obtain
$$
A^{1/4}A^{iy}B^{iy}B^{1/4}=\lambda B^{1/4}B^{iy}A^{iy}A^{1/4},
$$
but now both elements have the same norm and this shows that $\lambda=1$; then
$$
A^{1/4+iy}B^{1/4+iy}= B^{1/4+iy}A^{1/4+iy},
$$
and since $A,B>0$, the existence of analytic logarithms shows that again $A$ commutes with $B$. By symmetry, the same argument applies for any $z_0=3/4+iy$ in the other border of the strip.
\end{proof}

\begin{coro}
If $A$ does not commute with $B$, the inequality is strict:
$$
\tau(A^zB^tA^{1-z}B^{1-z}) < \tau(AB),
$$
in some open set $\Omega\subset \mathbb C$ containing the closed strip ${\mathcal S}_{1/4}$.
\end{coro}

\bigskip

If we allow $A,B$ to be non invertible, holomorphy is lost, but nevertheless in the same spirit we have the following result.
\begin{prop}\label{delta}
For given $A,B\ge 0$, there exists $\delta=\delta(A,B)>0$ such that
$$
\tau(A^{t}B^{t}A^{1-t}B^{1-t})\le \tau(AB)
$$
holds in the interval $[1/4-\delta,3/4+\delta]$. If $A$ does not commute with $B$, the inequality is strict in the whole $(1/4-\delta,3/4+\delta)$.
\end{prop}
\begin{proof}
If $A$ commutes with $B$, then the assertion is trivial. If not, arguing as in the last part of the proof of the previous proposition, we must have strict inequality
$$
\tau(A^{t}B^{t}A^{1-t}B^{1-t})<\tau(AB)
$$ 
for $t=1/4$, $t=3/4$, and then by continuity the inequality extends a bit out of the closed interval $[1/4,3/4]$. 

Consider $t\in (1/4,1/2)$ and put $X=B^{1/4}A^{1/4}A^{t-1/4}B^{t-1/4}$, $Y=B^{1/4}A^{1/4}A^{3/4-t}B^{3/4-t}$. Note that $\frac{1}{t},\frac{1}{1-t}\ge 1$ and define $1/p=t-1/4\in (0,1/4)$, $1/q=3/4-t\in (1/4,1/2)$, note also that $1/p+1/4=t$, $1/q+1/4=1-t$. By reiterated use of H\"older's inequality compute
\begin{eqnarray}
\tau(A^tB^tA^{1-t}B^{1-t})&\le & \|XY\|_1\le \|X\|_{t^{-1}}\|Y\|_{(1-t)^{-1}}\nonumber\\
&\le & \|B^{1/4}A^{1/4}\|_4 \|A^{1/p}B^{1/p}\|_p\|B^{1/q}A^{1/q}\|_q \|A^{1/4}B^{1/4}\|_4.\nonumber
\end{eqnarray}
Now apply Lemma \ref{arak} to each of the four terms (note that $p>4$ and $q>2)$, and we have\footnote{Note that this is another proof of the inequality for real $t\in [\frac14,\frac34]$.}
$$
\tau(A^tB^tA^{1-t}B^{1-t})\le \|B^{1/4}A^{1/4}\|_4 \|A^{1/p}B^{1/p}\|_p\|B^{1/q}A^{1/q}\|_q \|A^{1/4}B^{1/4}\|_4\le\tau(AB).
$$
If we assume equality of the traces, then
$$
\tau(AB)=\|B^{1/4}A^{1/4}\|_4 \|A^{1/p}B^{1/p}\|_p\|B^{1/q}A^{1/q}\|_q \|A^{1/4}B^{1/4}\|_4
$$
and in particular, it must be that $\|A^{1/4}B^{1/4}\|_4=\tau(AB)^{1/4}$, and from Lemma \ref{conmu} we can deduce that $A$ commutes with $B$. By the symmetry $(t\mapsto 1-t)$ the argument extends to $(1/2,3/4)$, and again by Lemma \ref{conmu} we already know that $A$ commutes with $B$ if equality is attained at $t=1/2$. This finishes the proof of the assertion that the inequality is strict in $[1/4,3/4]$ unless $A$ commutes with $B$.
\end{proof}

\begin{rem}
The inequalities in the previous proof give in fact
$$
\tau(|B^{\frac14}A^tB^tA^{1-t}B^{\frac34 -t}|)\le Tr(AB)
$$
for any $t\in [\frac14,\frac34]$; this is a particular instance of \cite[Theorem 2.10]{ando}.
\end{rem}

\section{Counterexamples}\label{contrajemplos}

In this section we  exhibit specific cases of different kind. In Example \ref{ex1} we choose $A,B$ such that $\|b_t(A,B)\|_\infty>\|h_t(A,B)\|_\infty$, while in Example \ref{ex2}, it is shown that the $j^{\hbox{th}}$ singular value of $A+B$ is not always greater than the $j^{\hbox{th}}$ singular value of $b_t(A,B)$. This provides negative answers to \cite[Conjecture 1.2]{kitta} and \cite[Problem 4]{audenart-kittaneh} respectively.

\begin{ejem} \label{ex1}

Consider the following positive definite matrices 
$$
A=\left(
\begin{array}{ccc}
 1141 & 0 & 0 \\
 0 & 204 & 0 \\
 0 & 0 & 1/8
\end{array}
\right) \qquad \hbox{ and } \qquad 
 B= \left(
\begin{array}{ccc}
 39 & 90 & 43 \\
 90 & 418 & 370 \\
 43 & 370 & 426
\end{array}
\right). 
$$
 
The following is the graph of $f(t)= -\|b_t(A,B)\|_\infty+\|h_t(A,B)\|_\infty$ for $t\in[0,\frac12]$:

\centerline{\includegraphics{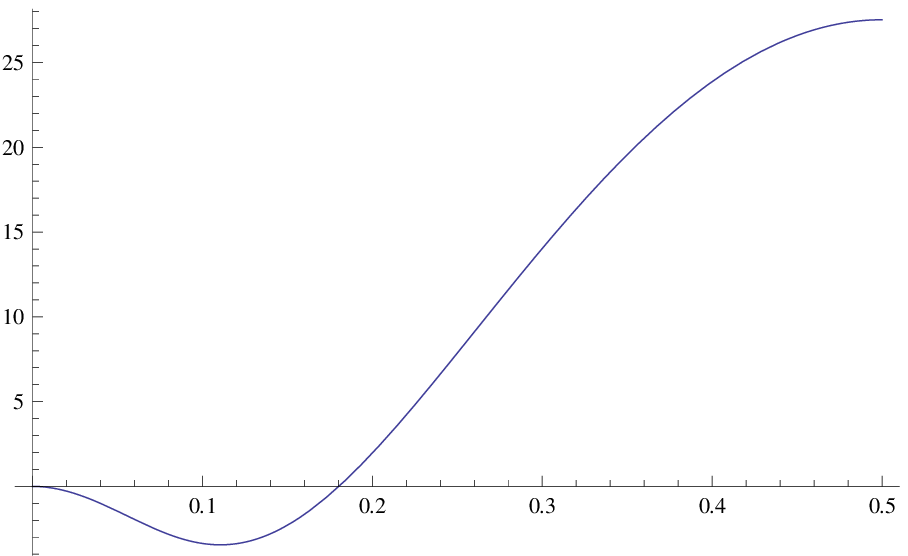}}

For these matrices $-\|b_t(A,B)\|_\infty+\|h_t(A,B)\|_\infty\simeq -2.3$ at $t=.15$.
\end{ejem}
 
In \cite[Problem 4]{audenart-kittaneh} K. Audenaert and F. Kittaneh asked if $s_j(b_t(A,B))\leq s_j(A+B)$ for every $j$ and $0<t<1$ (where $s_j(M)$, $j=1\dots n$ denote the singular values of the matrix $M$ arranged in non-increasing order).

\begin{ejem} \label{ex2}
Consider the following positive definite matrices 
  $$
  A= \left(
\begin{array}{ccc}
 6317 & 0 & 0 \\
 0 & 474 & 0 \\
 0 & 0 & 6
\end{array}
\right)
 \quad and  \quad 
B=\left(
\begin{array}{ccc}
 2078 & 2362 & 2199 \\
 2362 & 3267 & 2585 \\
 2199 & 2585 & 2492
\end{array}
\right).$$

Then, for $t=\frac12$ we have
$$s(b_{\frac{1}{2}}(A,B))=(6826.57,\ 878.499, \ 591.716)$$
and
$$s(A+B)=(10561.4,\ 3629.62,\ 443.017).$$

In particular, $s_3(b_{\frac12}(A,B))>s_3(A+B)$.
\end{ejem}

\bigskip

\end{document}